\documentclass[10pt]{article}
\usepackage{amsmath}
\usepackage{amsthm}
\usepackage{amsfonts}
\usepackage{mathrsfs}
\usepackage{amsfonts}
\usepackage{amssymb,amsfonts,amsmath, psfrag,eepic,colordvi,epsfig}
\usepackage{multirow} 
\usepackage{array} 

\usepackage{longtable}

\usepackage{cite}

\topskip 
\numberwithin{equation}{section}
\newtheorem{theorem}{Theorem}[section]

\newtheorem{definition}[theorem]{Definition}
\newtheorem{conjecture}[theorem]{Conjecture}

\newtheorem{lemma}[theorem]{Lemma}

\newcommand{\Rmnum}[1]{\uppercase\expandafter{\romannumeral #1}}

\begin{document}
	\parskip 7pt
	
	\pagenumbering{arabic}
	\def\sof{\hfill\rule{2mm}{2mm}}
	\def\ls{\leq}
	\def\gs{\geq}
	\def\SS{\mathcal S}
	\def\qq{{\bold q}}
	\def\MM{\mathcal M}
	\def\TT{\mathcal T}
	\def\EE{\mathcal E}
	\def\lsp{\mbox{lisp}}
	\def\rsp{\mbox{rasp}}
	\def\pf{\noindent {\it Proof.} }
	\def\mp{\mbox{pyramid}}
	\def\mb{\mbox{block}}
	\def\mc{\mbox{cross}}
	\def\qed{\hfill \rule{4pt}{7pt}}
	\def\pf{\noindent {\it Proof.} }
	\textheight=22cm

	\begin{center}
		{\Large\bf Proof  of  a conjecture
		 of Banerjee, Bringmann
		  and  Bachraoui
		   on infinite families of congruences }
	\end{center}

		\begin{center}

		Juejie Sun$^{1}$ and  Olivia X.M. Yao$^2$

			$^{1,2}$School of Mathematical Sciences, \\
			Suzhou University of Science and
			Technology, \\
			Suzhou,  215009, Jiangsu Province,
			P. R. China

			Email:   jeric$\_$020918@163.com,
			yaoxiangmei@163.com

		\end{center}

	\noindent {\bf Abstract.}
	Recently, Andrews and Bachraoui
	investigated congruences for 
	 certain restricted two-color partitions.
	  They made two conjectures for Ramanujan
	  type congruences and a vanishing identity for the limiting sequence.
	   Very recently,  Banerjee, Bringmann
	   and  Bachraoui confirmed  these three
	     conjectures  by relating the corresponding generating
	     function to modular forms and mock theta functions. 
	      At the end of their paper,  they  posed a conjecture
	       on infinite families of  congruences modulo 4 and 8 for  
	        the limiting sequence. The Banerjee-Bringmann-Bachraoui's
	          conjecture  implies the two  
	         conjectures given by Andrews and Bachraoui. In this note, we settle  Banerjee-Bringmann-Bachraoui's conjecture
	          on infinite famlies of congruences
	           based on   Banerjee-Bringmann-Bachraoui's results 
	            and an identity due to Waston.



	\section{Introduction}
	\allowdisplaybreaks
	
	A partition of a postive integer $n$ is a weakly decreasing sequence of positive integers $\pi=(\pi_{1},\pi_{2},\ldots,\pi_{k})$ such that
	$\pi_{1}+\pi_{2}+\cdots+\pi_{k}=n$. The $\pi_i$ are called the parts of partition. Let $l(\pi)$ and $\#(\pi)$ 
	denote the largest part of $\pi$
	and the number of parts of $\pi$
	\cite{Andrews-1976}.   
	
	In recent years, 
	 integer partitions in which 
	  each part may occur in two  colors have been studied ectensively, see 
	   for example \cite{Andrews-Bachraoui-1,Andrews-Bachraoui,Banerjee,
	   	Banerjee-Bringmann-Bachraoui,Chern}.  In particular,
	    Andrews and Bachraoui \cite{Andrews-Bachraoui} invesigated sequences
	     of integer partitions in two colors
	      (blue and red), as in the following definition.
	      
	      \begin{definition}\label{D-1}
	      	Let $k$ be a fixed positive integer.
	      	 For a positive integer $n$,
	      	  let $c(k,n)$
	      	   counts the number of 
	      	    two-color partitons $\pi$
	      	     of $n$ in which
	     \begin{enumerate}
	     	
	     	\item the smallest part $s(\pi) $
	     	 is odd and occurs at least one in blue,
	     	 
	     	 \item every even blue part is at least $2k-1$
	     	  greater than $s(\pi)$,
	     	  
	     	  \item the even parts of the same color are distinct. 
	     	
	     	  \end{enumerate}
	     	  	      \end{definition}
	     	  
	     	Andrews and Bachraoui \cite{Andrews-Bachraoui}  
	     	established the generating function of $c(k,n)$
	     	\begin{align*}
	     		C_k(q):=\sum_{n=0}^\infty c(k,n)q^n
	     		=\sum_{n=0}^\infty
	     		\frac{(-q^{2n+2k},-q^{2n+2};q^2)_\infty
	     		 }{(q^{2n+1};q^2)_\infty^2} q^{2n+1}.
	     	\end{align*}
	     	Here and throughout, we adopt the following standard
	     	$q$-series notation:
	     	\begin{align*}
	     		(a;q)_\infty:=&\prod_{k=0}^\infty
	     		(1-aq^k),\qquad
	     		(a;q)_n:=\frac{(a;q)_\infty}
	     		{(aq^n;q)_\infty}
	     	\end{align*}
	     	and
	     	\[
	     	(a_1,a_2,\ldots, a_r;q)_\infty
	     	:=(a_1;q)_\infty(a_2;q)_\infty
	     	\cdots (a_r;q)_\infty.
	     	\]
	 Andrews and Bachraoui \cite{Andrews-Bachraoui}  proved a number 
	  of congruences modulo 4 on  $c(k,n)$
	   for  $k=1,2,3$. At the end of their paper,
	    they considered the congruence of the limits 
	    of $C_k(q)$. Let 
	    \begin{align*} 
	    C(q):=\lim_{k  \to  \infty}C_k(q)=\sum_{n=0}^\infty
	     \frac{(-q^{2n+2};q^2)_\infty
	     }{(q^{2n+1};q^2)_\infty^2} q^{2n+1}=:\sum_{n=0}^\infty
	      c(n) q^n.
	    \end{align*}
	   Andrews and Bachraoui \cite{Andrews-Bachraoui}  posed the following two conjectures.
	   
	\begin{conjecture} \label{C-1}
		For $n\geq 0$,
		 \begin{align}\label{1-2}
		 c(8n+4) \equiv 0\pmod 4.
		 \end{align}
	\end{conjecture}
	
		\begin{conjecture} \label{C-2}
		For $n\geq 0$,
		\begin{align}\label{1-3}
			c(8n+6) \equiv 0\pmod 8.
		\end{align}
	\end{conjecture}
	
Very recently, Banerjee,   Bringmann and    Bachraoui
 \cite{Banerjee-Bringmann-Bachraoui} proved Conjectures 
 \ref{C-1} and \ref{C-2}  by relating the corresponding generating
 function to modular forms and mock theta functions.
 They also showed  that for $n\geq 0$,
 \begin{align}\label{1-4}
 	c(16n+13) \equiv 0\pmod 4. 
 \end{align}
  
 At the end of their paper \cite{Banerjee-Bringmann-Bachraoui}, Banerjee,   Bringmann and    Bachraoui
  presented 
  the following two conjectures.
  
  \begin{conjecture} \label{C-3}
  	For $n\geq 0$,
 \begin{align}\label{1-5}
  c(32n+23) \equiv 0 \pmod 8.
  \end{align}
  \end{conjecture}
  
  \begin{conjecture} \label{C-4}
  	For all integers $n$ and $k$,
  	\begin{align}
  		c\left(2^{2k+3}n+\frac{11\cdot 4^{k}+1}{3}\right) &
  		 \equiv 0 \pmod 4, \label{1-6} \\
  		 	c\left(2^{2k+3}n+\frac{17\cdot 4^{k}+1}{3}\right) &
  		 \equiv 0 \pmod 8,  \label{1-7} \\
  		 	c\left(2^{2k+4}n+\frac{38 \cdot 4^{k}+1}{3}\right) &
  		 \equiv 0 \pmod 4. \label{1-8}
  	\end{align}

  \end{conjecture}
  
  Setting $k=0$ in \eqref{1-6}--\eqref{1-8}, we arrive at \eqref{1-2}--\eqref{1-4}, respectively.
   Taking $k=1$ in \eqref{1-7}, we get \eqref{1-5}. Therefore, Conjecture  \ref{C-4} implies 
    Conjectures \ref{C-1}--\ref{C-3}.

The aim of this paper is to confirm Conjecture
\ref{C-4}  by utilizing    Banerjee-Bringmann-Bachraoui's results 
  \cite{Banerjee-Bringmann-Bachraoui} and an identity due to Waston \cite{Watson}.

\section{Proof of Conjecture \ref{C-4} }

To prove Conjecture \ref{C-4}, we first 
 prove two lemmas. 

\begin{lemma}\label{L-1}
	For all nonnegative integers $n$ and $k$, 
\begin{align}
c\left(2^{2k+2}n+\frac{2^{2k+3}+1}{3} \right)  \equiv (-1)^k c(4n+3) \pmod 8. \label{2-1}
\end{align}
\end{lemma}

\noindent{\it Proof.} In 
 \cite{Banerjee-Bringmann-Bachraoui}, Banerjee, Bringmann and Bachraoui proved that 
 \begin{align} \label{2-2}
 	C(q)=2q\frac{f_2f_4
 	 }{f_1^2}B(-q)-q\omega(-q).
 \end{align}
 Here and throughout, 
 for any psotive integer $k$, 
 \[
 f_m:=(q^k;q^k)_\infty. 
 \]
  The third order mock theta function 
  $
  \omega(q)$  is   defined by 
   \[
   \omega(q):=\sum_{n=0}^\infty
    \frac{q^{2n(n+1)}}{(q;q^2)_{n+1}^2}
   \]
   and 
    McIntosh's  second order mock theta 
   function $B(q)$ \cite{McIntosh} is defined by
   \begin{align} \label{2-3}
   	B(q):=\sum_{n=0}^\infty a_B(n)q^n=
   	\frac{(-q^2;q^2)_\infty 
   	}{(q^2;q^2)_\infty }
   	\sum_{n=-\infty}^\infty
   	\frac{(-1)^n q^{2n(n+1)}}{1-q^{2n+1}}
   	=\sum_{n=0}^\infty\frac{(-q^2;q^2)_n}{(q;q^2)_{n+1}^2} q^{n(n+1)}.
   \end{align}
   In \cite{Watson}, Watson proved that 
\begin{align} \label{2-4}
   f(q^8)-2q\omega(-q)-2q^3\omega(-q^4)=\frac{f_1^2f_4^8}{f_2^5f_8^4
   }.
   \end{align}
   Combining \eqref{2-2}--\eqref{2-4} yields 
   \begin{align}
   	\sum_{n=0}^\infty c(n)q^n
   	&=2q\frac{f_2f_4}{f_1^2}
   	\sum_{n=0}^\infty (-1)^n a_B(n) q^n
   	+q^3\omega(-q^4)+\frac{f_1^2f_4^8}{2f_2^5f_8^4
   	}-\frac{f(q^8)}{2}
 . \label{2-5}
   \end{align}
It follows from \cite[Entry 25, (i) and (ii), p.40]{Berndt} that 
\begin{align}\label{2-6}
\frac{1}{f_1^2}=\frac{f_8^5}{f_2^5f_{16}^2}
+2q\frac{f_4^2f_{16}^2}{f_2^5f_8}
\end{align}
and 
\begin{align}\label{2-7}
f_1^2=\frac{f_2f_8^5}{f_4^2f_{16}^2}
-2q\frac{f_2f_{16}^2}{f_8}.
\end{align}
Substituting \eqref{2-6}
and \eqref{2-7} into \eqref{2-5}, we obtain  
\begin{align}\label{2-8}
	\sum_{n=0}^\infty c(n)q^n
	  &=2qf_2f_4\left(
	  \frac{f_8^5}{f_2^5f_{16}^2}+2q\frac{f_4^2f_{16}^2
	   }{f_2^5f_8}\right)\left(\sum_{n=0}^\infty a_B(2n)q^{2n}
	 -\sum_{n=0}^\infty a_B(2n+1)q^{2n+1} 
	   \right)
	   \nonumber\\
	   &\qquad +q^3\omega(-q^4)-\frac{f(q^8)}{2}
	   +\frac{f_4^8}{2f_2^5f_8^4}
	   \left(\frac{f_2f_8^5}{f_4^2f_{16}^2}
	   -2q\frac{f_2f_{16}^2}{f_8}\right).
\end{align}
 Picking out those terms in which the
power of $q$
is congruent to 1 modulo 2 in \eqref{2-8}, after dividing
by $q $ and
replacing $q^2$ by $q$, we arrive at 
	\begin{align}\label{2-9}
		\sum_{n=0}^\infty
		 c(2n+1)q^n =2\frac{f_2f_4^5}{
		  f_1^4f_8^2}\sum_{n=0}^\infty a_B(2n)q^n 
		  -4q\frac{f_2^3f_8^2}{f_1^4f_4}
		  \sum_{n=0}^\infty a_B(2n+1)q^n +q\omega(-q^2)
		  -\frac{f_2^8f_8^2}{f_1^4f_4^5}.
	\end{align}
Mao \cite{Mao} proved that 
	\begin{align}\label{2-10}
\sum_{n=0}^\infty a_B(2n)q^n=\frac{f_2^5}{f_1^4}.
\end{align}
Wang \cite{Wang} proved that 
\[
\sum_{n=0}^\infty a_B(n)q^n\equiv q^{2n^2+2n} \pmod 2,
\]
which implies that 
	\begin{align}\label{2-11}
	\sum_{n=0}^\infty a_B(2n+1)q^n\equiv 0 \pmod 2. 
	\end{align}
	Combining \eqref{2-9}--\eqref{2-11} yields 
	\begin{align}
	\sum_{n=0}^\infty
	c(2n+1)q^n & \equiv 2\frac{f_2^6f_4^5}{
		f_1^8f_8^2}  +q\omega(-q^2)
	-\frac{f_2^8f_8^2}{f_1^4f_4^5}
	\nonumber\\
	&\equiv  2\frac{f_2^2f_4^5}{
		 f_8^2}  +q\omega(-q^2)
	-\frac{f_2^8f_8^2}{f_1^4f_4^5} \pmod 8.
	\label{2-12}
\end{align}	
Here we have used the fact that for all positive integers $m$ and $k$,
\begin{align}\label{2-13}
f_k^{2^m} \equiv f_{2k}^{2^{m-1}} \pmod {2^m}.
\end{align}
It follows from \cite[Entry 25, (v) and (vi), p.40]{Berndt} that 
\begin{align}\label{2-14}
\frac{1}{f_1^4}=\frac{f_4^{14}}{f_2^{14}f_8^4}
+4q\frac{f_4^2f_8^4}{f_2^{10}}.
\end{align}
Substituting \eqref{2-14} into \eqref{2-12} and then multiplying 
 both sides by $q$, we
 have 
\begin{align}\label{2-15}
	\sum_{n=0}^\infty
	c(2n+1)q^{n+1} 
	&\equiv  2q\frac{f_2^2f_4^5}{
		f_8^2}  +q^2\omega(-q^2)
	-q\frac{f_2^8f_8^2}{ f_4^5}\left(
	\frac{f_4^{14}}{f_2^{14}f_8^4}+4q\frac{f_4^2f_8^4}{
	f_2^{10}}\right) \pmod 8.
\end{align}	
 Picking out those terms in which the
power of $q$
is congruent to 0  modulo 2 in \eqref{2-15}, 
 and then replacing $q^2$ by $q$,  we obtain 
\begin{align}\label{2-16}
	\sum_{n=0}^\infty
	c(4n+3)q^{n+1} 
	&\equiv  q\omega(-q)
	-4q\frac{f_4^6}{ f_1^2f_2^3} \nonumber\\
	&  \equiv  q\omega(-q)
	-4qf_4^4 \pmod 8. \qquad ({\rm by}\ \eqref{2-13})
\end{align}	
It follows from \eqref{2-2} that 
\begin{align}\label{2-17}
q\omega(-q)=2q\frac{f_2f_4}{f_1^2}\sum_{n=0}^\infty
 (-1)^n a_B(n)q^n -\sum_{n=0}^\infty c(n)q^n .
\end{align}
 Substituting \eqref{2-17}
  into \eqref{2-16}, we arrive at 
\begin{align}\label{2-18}
	\sum_{n=0}^\infty
	c(4n+3)q^{n+1} 
	&\equiv 2q\frac{f_2f_4}{f_1^2}\sum_{n=0}^\infty
	(-1)^n a_B(n)q^n -\sum_{n=0}^\infty c(n)q^n 
	-4q f_4^4 \nonumber
	\\
	&\equiv 2qf_2f_4\left(
	\frac{f_8^5}{f_2^5f_{16}^2}+2q\frac{f_4^2f_{16}^2
	}{f_2^5f_8}\right)\left(\sum_{n=0}^\infty a_B(2n)q^{2n}
	-\sum_{n=0}^\infty a_B(2n+1)q^{2n+1} 
	\right)\nonumber\\
	&\qquad  -\sum_{n=0}^\infty c(n)q^n 
	-4q f_4^4 
	 \pmod 8. \qquad ({\rm by}\ \eqref{2-6})
\end{align}	
Extracting those
terms in which the power of  $q $ is congruent to 1 modulo 2
 in \eqref{2-18}, then dividing by $q$ and 
replacing $q^{2}$ by $q$, we obtain
	\begin{align}
	\sum_{n=0}^\infty
	c(8n+3)q^{n} 
	&\equiv 2\frac{f_2f_4^5}{f_1^4f_8^2}\sum_{n=0}^\infty
  a_B(2n)q^n -4q\frac{f_2^3f_8^2}{
	f_1^4f_4}\sum_{n=0}^\infty
	(-1)^n a_B(2n+1)q^n
	\nonumber\\
	&\qquad -\sum_{n=0}^\infty c(2n+1)q^n 
	-4f_2^4  \nonumber\\
		&\equiv 2\frac{f_2^6f_4^5}{f_1^8f_8^2} 
 -\sum_{n=0}^\infty c(2n+1)q^n -4f_2^4
 \qquad ({\rm by}\ \eqref{2-10}\ {\rm and }\ \eqref{2-11})	\nonumber\\
 &\equiv 6\frac{f_2^2f_4^5}{ f_8^2} 
 -\sum_{n=0}^\infty c(2n+1)q^n 
   \pmod 8. \qquad ({\rm by}\ \eqref{2-13}) \label{2-18-1}
\end{align}	
It follows from \eqref{2-18-1}
 that  for $n\geq 0$,
\begin{align}\label{2-19}
	c(16n+11) \equiv -c(4n+3) \pmod 8. 
\end{align}
By \eqref{2-19} and mathematial induction, we obtain \eqref{2-1}.
\qed  

\begin{lemma}
	For $n\geq 0$,
	\begin{align}
		c(32n+15)&\equiv 0 \pmod 4,
		\label{2-21}\\
		c(32n+23)&\equiv 0 \pmod 8, \label{2-22}\\
		c(64n+51)&\equiv 0 \pmod 4. \label{2-23}
	\end{align}
\end{lemma}

\noindent{\it Proof.}
Extracting those
terms in which the power of  $q $ is congruent to 0 modulo 2
in \eqref{2-18}, then  
replacing $q^{2}$ by $q$, we obtain
\begin{align}\label{a-1}
	\sum_{n=0}^\infty c(8n+7)q^{n+1}
	&\equiv 4q\frac{f_2^3f_8^2}{f_1^4f_4}
	\sum_{n=0}^\infty a_B(2n)q^n-2q\frac{f_2f_4^5}{f_1^4f_8^2}
	\sum_{n=0}^\infty a_B(2n+1)q^n-\sum_{n=0}^\infty c(2n)q^n 
	\nonumber\\
		&\equiv 4q f_4f_8^2 
	-2q\frac{f_4}{f_2}
	\sum_{n=0}^\infty a_B(2n+1)q^n\nonumber\\
	&\quad -\sum_{n=0}^\infty c(2n)q^n 
	\pmod 8,  \qquad ({\rm by}\ \eqref{2-10}\ {\rm and }\ \eqref{2-13})
\end{align}
from which with \eqref{2-11}, we get 
\begin{align}
	\sum_{n=0}^\infty c(8n+7)q^{n+1}
	& \equiv -\sum_{n=0}^\infty c(2n)q^n 
	\pmod 4. \label{2-24}
\end{align}
It follows from \eqref{2-24}
 that for $n\geq 0$
\begin{align}\label{2-25}
c(8n+7) \equiv - c(2n+2) \pmod 4. 
\end{align}
Replacing $n$ by $4n+1$ in \eqref{2-25} and using \eqref{1-2}, we arrive at \eqref{2-21}. 

In \cite{Chan-Mao},
  Chan and Mao proved that 
 \begin{align}
 		\sum_{n=0}^\infty a_B(4n+1)q^n=2\frac{f_2^8}{f_1^7}.
 		 \label{a-2}
 \end{align}
Extracting those
terms in which the power of  $q $ is congruent to 1 modulo 2
in \eqref{a-1}, then  divicing by $q$ and 
replacing $q^{2}$ by $q$, we obtain
\begin{align*}
	\sum_{n=0}^\infty c(16n+7)q^{n}
	&\equiv 4f_2f_4^2 
	-2\frac{f_2}{f_1}
	\sum_{n=0}^\infty a_B(4n+1)q^n-\sum_{n=0}^\infty c(4n+2)q^n 
	\nonumber\\
		&\equiv 4f_2f_4^2 
	-4\frac{f_2^9}{f_1^8}
 -\sum_{n=0}^\infty c(4n+2)q^n  \qquad ({\rm by}\ \eqref{a-2})
 	\nonumber\\
 &\equiv  
 -\sum_{n=0}^\infty c(4n+2)q^n 
	\pmod 8, \qquad ({\rm by}\ \eqref{2-13})
\end{align*}
which implies that for $n\geq 0$, 
\begin{align}\label{2-26}
c(16n+7) \equiv -c(4n+2) \pmod 8.
\end{align}
Replacing $n$ by $2n+1$ in \eqref{2-26} and using \eqref{1-3}, we obtain  \eqref{2-22}. 

Picking out  those
terms in which the power of  $q $ is congruent to 0 modulo 2
in \eqref{2-18-1}, then   
replacing $q^{2}$ by $q$, we get  
\begin{align*}
	\sum_{n=0}^\infty
	c(16n+3)q^{n} 
	&\equiv 6\frac{f_1^2f_2^5}{ f_4^2} 
	-\sum_{n=0}^\infty c(4n+1)q^n 
	\nonumber\\
	&\equiv 2 f_4
	-\sum_{n=0}^\infty c(4n+1)q^n 
	\pmod 4, \qquad ({\rm by}\ \eqref{2-13})
\end{align*}
from which, we deduce that 
 for $n\geq 0$, 	
\begin{align}
 	c(32n+19) \equiv 
	-  c(8n+5) 
	\pmod 4. \label{2-27}
\end{align}	
Replacing $n$ by $2n+1$
 in \eqref{2-27}
 and using \eqref{1-4}, we get \eqref{2-23}.  \qed 
  
  Now, we turn to prove Conjecture \ref{C-4}.
  
 Replacing $n$ by $8n+3$ in \eqref{2-1}
  and employing \eqref{2-21}, we see that 
   for $n,k\geq 0$,
   \begin{align*}
   	c\left(2^{2k+5}n+\frac{11\cdot 4^{k+1}+1}{3} \right)  \equiv  0 \pmod 4. 
   \end{align*} 
   The above congrence
    implies that \eqref{1-6} is true when $k\geq 1$. 
   Congruence  \eqref{1-2}
    implies  \eqref{1-6} holds when $k=0$. Therefore, 
     Congruence  \eqref{1-6} is true for $k\geq 0$.  
    
     Replacing $n$ by $8n+5$ in \eqref{2-1}
    and employing \eqref{2-22}, we see that 
    for $n,k\geq 0$,
    \begin{align*}
    	c\left(2^{2k+5}n+\frac{17\cdot 4^{k+1}+1}{3} \right)  \equiv  0 \pmod 8,
    \end{align*}
    from which with \eqref{1-3}, we arrive at \eqref{1-7}.
    
     Replacing $n$ by $16n+12$ in \eqref{2-1}
    and employing \eqref{2-23}, we see that 
    for $n,k\geq 0$,
    \begin{align*}
    	c\left(2^{2k+6}n+\frac{38\cdot 4^{k+1}+1}{3} \right)  \equiv  0 \pmod 4.
    \end{align*}
   Congrence \eqref{1-8}
    follows  from the above congruence and \eqref{1-4}. 
     This completes the proof of Conjecture \ref{C-4}. \qed
     
      \section{Conclusions}
     
     As seen in Introduction, integer partitions
     in which each part may occur in two colors  
     have received
     a lot of attention in recent years.  In this note, we prove
      a conjecture   
      posed by  Banerjee, Bringmann
      and  Bachraoui \cite{Banerjee-Bringmann-Bachraoui}
      on infinite families of congruences
      for the coefficients of $C(q)$. Banerjee-Bringmann-Bachraoui's
       conjecture implies two conjectures due to Andrews 
       and Bachraoui \cite{Andrews-Bachraoui-1}.
       A natural problem
     is to extend
     the congruences in
     this paper  to modulo $16$,
     $32$, etc.
     In addition, it would be
     interesting
     to determine the
     arithmetic density of   the
     set of
     integers such that
     $c(n)\equiv 0 \pmod
     {2^k}$ for some  fixed positive
     integers $k$.

\end{document}